\documentclass{amsart}
\usepackage{amsmath}
\usepackage{amsfonts}

\newtheorem{theorem}{Theorem}
\theoremstyle{plain}
\newtheorem{innercustomthm}{Theorem}
\newenvironment{customthm}[1]
  {\renewcommand\theinnercustomthm{#1}\innercustomthm}
  {\endinnercustomthm}

\newtheorem{corollary}{Corollary}

\newtheorem{definition}{Definition}

\newtheorem{proposition}{Proposition}

\numberwithin{equation}{section}

\begin{document}
\title[ logharmonic mappings ]{On geometrical properties of starlike
logharmonic mappings of order $\alpha $}
\author{ Zayid AbdulHadi}
\address{Department of Mathematics and statistics\\
American University of Sharjah\\
Sharjah, Box 26666\\
UAE}
\email{zahadi@aus.edu}
\author{ Layan El Hajj}
\address{Department of Mathematics, American University of \ Dubai, Dubai,\\
UAE. }
\email{ lhajj@aud.edu }
\date{June 26, 2016}
\subjclass[2000]{Primary 30C35, 30C45; Secondary 35Q30.}
\keywords{logharmonic mappings; rotationally starlike mappings; radius of
starlikeness; distortion estimate. }

\begin{abstract}
In this paper, we find the radius of the disk $\Omega _{r}$ such that every
starlike logharmonic mapping $f(z)$ of order $\alpha ,$ is starlike in $%
|z|\leq r$ with respect to any point of $\Omega _{r}.$ We also establish a
relation between the set of starlike logharmonic mappings \ and the set of
starlike logharmonic mappings of order alpha. Moreover, the radius of
starlikeness and univalence for the set of close to starlike logharmonic
mappings of order $\alpha $ is determined.
\end{abstract}

\maketitle

\section{Introduction}

Let $H(U)$ be the linear space of all analytic functions defined in the unit
disk $U=\{z:|z|<1\}$ of the complex plane $%
\mathbb{C}
$ and let $B$ denote the set of functions $a\in $ $H(U)$ satisfying $%
|a(z)|<1 $ in $U.$ A logharmonic mapping defined on $U$ is a solution of the
nonlinear elliptic partial differential equation%
\begin{equation}
\frac{\overline{f_{\overline{z}}}}{\overline{f}}=a\frac{f_{z}}{f},
\label{eq1.1}
\end{equation}%
where the second dilatation function $a$ belongs to the class $B$. Thus the
Jacobian

\begin{equation*}
J_{f}=\left\vert f_{z}\right\vert ^{2}(1-|a|^{2})
\end{equation*}%
is positive and hence, all non-constant logharmonic mappings are
sense-preserving and open on $U$. If $f$ is a non-constant logharmonic
mapping of $U$ and vanishes only at $z=0$, then $f$ admits the
representation 
\begin{equation}
f(z)=z^{m}|z|^{2\beta m}h(z)\overline{g(z)},  \label{eq1.2}
\end{equation}%
where $m$ is a nonnegative integer, $\mathrm{Re}(\beta )>-1/2$, and $h$ and $%
g$ are analytic functions in $U$ satisfying $g(0)=1$ and $h(0)\neq 0$ (see
[1]). The exponent $\beta $ in (1.2) depends only on $a(0)$ and can be
expressed by

\begin{equation*}
\beta =\overline{a(0)}\dfrac{1+a(0)}{1-|a(0)|^{2}}.
\end{equation*}

Note that $f(0)\neq 0$ if and only if $m=0$, and that a univalent
logharmonic mapping on $U$ vanishes at the origin if and only if $m=1$, that
is, $f$ has the form

\begin{equation*}
f(z)=z|z|^{2\beta }h(z)\overline{g(z)},
\end{equation*}%
where $\mathrm{Re}(\beta )>-1/2$ and $0\notin (hg)(U).$ This class has been
studied extensively in recent years, for instance in [1, 2, \ 3, 4, 5, 6, 7,
8, 9, 11, 12, 13, 21, 22, 26] .

As further evidence of its importance, note that $F(\zeta )=\log f(e^{\zeta
})$ are univalent harmonic mappings of the half-plane $\{\zeta :\mathrm{Re}%
(\zeta )<0\}$. Studies on univalent harmonic mappings can be found in [10,
13, 14, 15, 16, 17, 18, 19, 20]. Such mappings are closely related to the
theory of minimal surfaces (see [24, 25]).

When $f$ is a nonvanishing logharmonic mapping in $U$, it is known that $f$
can be expressed as 
\begin{equation*}
f(z)=h(z)\overline{g(z)},
\end{equation*}%
where $h$ and $g$ are nonvanishing analytic functions in $U.$

Let $f=zh(z)\overline{g(z)}$ be a univalent logharmonic mapping. We say that 
$f$ is a starlike logharmonic mapping of order $\alpha $ if 
\begin{equation*}
\dfrac{\partial \arg f(re^{i\theta })}{\partial \theta }=\Re \dfrac{zf_{z}-%
\overline{z}f_{\overline{z}}}{f}>\alpha \text{, }0\leq \alpha <1\text{ }
\end{equation*}
for all $z\in U.$ Denote by $ST_{Lh}(\alpha )$ the set of all starlike
logharmonic mappings of order $\alpha .$ If $\alpha =0,$ we get the class of
starlike logharmonic mappings. We also denote $ST(\alpha )=\{f\in $ $%
ST_{Lh}(\alpha )$ and $f\in H(U)\}.$ A detailed study of the class $%
ST_{Lh}(\alpha )$ to be found in [4]. In particular, the following are
representation theorem and distortion theorem for mappings in the set $%
ST_{Lh}(\alpha )$.

\begin{customthm}{A}\label{A}
(Representation Theorem) Let $f(z)=zh(z)\overline{g(z)}\ $be$\ $a$\ $%
logharmonic mapping on $U$, $0\notin hg(U).$ Then $f\in ST_{Lh}(\alpha )$ if
and only if \ $\varphi (z)=zh(z)/g(z)\in ST(\alpha )$ \ and it follows that%
\begin{equation*}
f(z)=\varphi (z)\exp 2\Re \int_{0}^{z}\dfrac{a(s)\varphi ^{\prime }(s)}{%
\varphi (s)(1-a(s))}ds.
\end{equation*}
\end{customthm}

\begin{customthm}{B}\label{B}
(Distorsion Theorem) Let $f(z)=zh(z)\overline{g(z)}\in ST_{Lh}(\alpha )$
with $a(0)=0$. Then for $z\in U\ $we have

\begin{equation*}
\dfrac{|z|}{(1+|z|)^{2\alpha }}\exp \left( (1-\alpha )\dfrac{-4|z|}{1+|z|}%
\right) \leq |f(z)|\leq \dfrac{|z|}{(1-|z|)^{2\alpha }}\exp \left( (1-\alpha
)\dfrac{4|z|}{1-|z|}\right) .
\end{equation*}

The equalities occur if and only if $f(z)=\overline{\zeta }f_{0}(\zeta z),$

\begin{equation*}
f_{0}(z)=\dfrac{z(1-\overline{z})}{(1-z)}\dfrac{1}{(1-\overline{z})^{2\alpha
}}\exp (1-\alpha )\Re \dfrac{4z}{1-z}.
\end{equation*}
\end{customthm}

Denote by $P_{Lh}$ the set of all logharmonic mappings $R$ defined on the
unit disk $U$ which are of the form $R=H\overline{G},$ where $H$ and $G$ are
in $H(U)$, $H(0)=G(0)=1$ and such that $Re\left( R(z)\right) >0$ for all $%
z\in U$. In particular, the set $P$ of all analytic functions $p(z)$ in $U$
with $p(0)=1$ and $Re(p(z))>0$ in $U$ is a subset of $P_{Lh}$ (for more
details see[2]).

In Section 2, we consider a relation between $ST_{Lh}(\alpha )$ and $%
ST_{Lh}(0)$ and obtain the radius of the disk $\Omega _{r}$ such that every
starlike logharmonic mapping $f(z)$ of order $\alpha ,$ is starlike in $%
|z|<r $ with respect to any point of $\Omega _{r}$. In section 3, the radius
of univalence and starlikeness is determined for the set of close to
starlike logharmonic mappings of order $\alpha .$

\section{Geometrical properties of the class $ST_{Lh}(\protect\alpha )$}

In the following two propositions we establish a relationship between the
classes $ST_{Lh}(\alpha )$ and $ST_{Lh}(0).$

\begin{proposition}
Let $f(z)=zh(z)\overline{g(z)}\in ST_{Lh}(\alpha )$ with respect to $a\in B$
and  
\begin{equation*}
K(z)=z\exp 2\Re \int_{0}^{z}\dfrac{a(z)}{1-a(z)}\dfrac{dz}{z}.
\end{equation*}%
Then $F(z)=f(z)^{\tfrac{1}{1-\alpha }}K(z)^{\tfrac{-\alpha }{1-\alpha }}\in
ST_{Lh}(0).$
\end{proposition}

\begin{proof}
Let $f(z)=zh(z)\overline{g(z)}\in ST_{Lh}(\alpha )$ with respect to $a\in B$
and  $$K(z)=z\exp 2\Re \int_{0}^{z}\dfrac{a(z)}{1-a(z)}\dfrac{dz}{z}\ .$$

We consider the function 
\begin{equation*}
F(z)=f(z)^{\tfrac{1}{1-\alpha }}K(z)^{\tfrac{-\alpha }{1-\alpha }}.
\end{equation*}%
Direct calculations yield 
\begin{equation*}
\dfrac{\overline{F_{\overline{z}}}}{\overline{F}}=\dfrac{1}{1-\alpha }\dfrac{%
\overline{f_{\overline{z}}}}{\overline{f}}+\dfrac{-\alpha }{1-\alpha }\dfrac{%
\overline{K_{\overline{z}}}}{\overline{K}}=\dfrac{1}{1-\alpha }a\dfrac{f_{z}%
}{f}+\dfrac{-\alpha }{1-\alpha }a\dfrac{K_{z}}{K}=a\dfrac{F_{z}}{F}.
\end{equation*}%
Hence $F$ is logharmonic with respect to the same $a$. Moreover, 
\begin{equation*}
\Re \dfrac{zF_{z}-\overline{z}F_{\overline{z}}}{F}=\dfrac{1}{1-\alpha }\ \Re 
\dfrac{zf_{z}-\overline{z}f_{\overline{z}}}{f}+\dfrac{-\alpha }{1-\alpha }%
\Re \dfrac{zK_{z}-\overline{z}K_{\overline{z}}}{K}>\dfrac{\alpha }{1-\alpha }%
+\dfrac{-\alpha }{1-\alpha }=0.
\end{equation*}%
Thus, $F\in ST_{Lh}(0).$
\end{proof}

\begin{proposition}
Let $f(z)=zh(z)\overline{g(z)}\in ST_{Lh}(0)$ with respect to $a\in B$ and 
\begin{equation*}
K(z)=z\exp 2\Re \int_{0}^{z}\dfrac{a(z)}{1-a(z)}\dfrac{dz}{z}.
\end{equation*}
Then $F(z)=f(z)^{1-\alpha }K(z)^{\alpha }\in ST_{Lh}(\alpha ).$
\end{proposition}

\begin{proof}
Let $f(z)=zh(z)\overline{g(z)}\in ST_{Lh}(0)$ with respect to $a\in B$ and

\begin{equation*}
K(z)=z\exp 2\Re \int_{0}^{z}\dfrac{a(z)}{1-a(z)}\dfrac{dz}{z}.
\end{equation*}

Consider the function 
\begin{equation*}
F(z)=f(z)^{1-\alpha }K(z)^{\alpha }.
\end{equation*}

Straightforward calculations give that

\begin{equation*}
\dfrac{\overline{F_{\overline{z}}}}{\overline{F}}=(1-\alpha )\dfrac{%
\overline{f_{\overline{z}}}}{\overline{f}}+\alpha \dfrac{\overline{K_{%
\overline{z}}}}{\overline{K}}=(1-\alpha )a\dfrac{f_{z}}{f}+\alpha a\dfrac{%
K_{z}}{K}=a\dfrac{F_{z}}{F},
\end{equation*}

and

\begin{equation*}
\Re \dfrac{zF_{z}-\overline{z}F_{\overline{z}}}{F}=(1-\alpha )\ \Re \dfrac{%
zf_{z}-\overline{z}f_{\overline{z}}}{f}+\alpha \Re \dfrac{zK_{z}-\overline{z}%
K_{\overline{z}}}{K}>\alpha .
\end{equation*}

Hence $F\in ST_{Lh}(\alpha )$ with respect to the same $a.$
\end{proof}

In what follows next, our objective is to find the region $\Omega _{r}$ in
the $w-$plane such that every $f\in ST_{Lh}(\alpha )$ is starlike with
respect to any point of $\Omega _{r}.$ Since $ST_{Lh}(\alpha )$ is compact
(see[4]), it follows that $\Omega _{r}$ is a closed set. Therefore, $\Omega
_{r}$ is a closed disk with center at $w=0$ and the determination of $\Omega
_{r}$ is equivalent to the determination of the radius of the disk $\Omega
_{r}$.

Our main result is the following theorem

\begin{theorem}
\bigskip Let $f\in ST_{Lh}(\alpha )$, then the radius of the disk $\Omega
_{r}$ such that $f$ is starlike with respect to any point of $\Omega _{r}$
is given by%
\begin{equation*}
\lambda _{\alpha }(r_{0})=\dfrac{r_{0}}{(1+r_{0})^{2\alpha }}\exp \left( 
\dfrac{-(1-\alpha )4r_{0}}{(1+r_{0})}\right) \dfrac{\left( \alpha +(1-\alpha
)\dfrac{1-r_{0}}{1+r_{0}}\right) }{\left( \alpha +(1-\alpha )\left( \dfrac{%
1+r_{0}}{1-r_{0}}\right) \right) \left( \dfrac{1+r_{0}}{1-r_{0}}\right) },
\end{equation*}%
where $r_{0}\in (0,1)$ and $r_{0}$ is the smallest positive root of the
equation

$8r^{5}\alpha ^{3}-12r^{5}\alpha ^{2}+6r^{5}\alpha -r^{5}-16r^{4}\alpha
^{3}+12r^{4}\alpha ^{2}+4r^{4}\alpha -3r^{4}+8r^{3}\alpha ^{3}-36r^{3}\alpha
^{2}+32r^{3}\alpha -8r^{3}+4r^{2}\alpha ^{2}-4r^{2}\alpha +4r^{2}-6r\alpha
+9r-1=0.$
\end{theorem}

\begin{proof}
Let $U_{r}(f)=f(|z|\leq r<1),$ $w=f(z)\in ST_{Lh}(\alpha ).$ $U_{r}(f)$ is
starlike with respect to $w_{0}$ if and only if

\begin{equation*}
\dfrac{\partial \arg \left( f(re^{i\theta })-w_{0}\right) }{\partial \theta }%
=\Re \dfrac{zf_{z}(z)-\overline{z}f_{\overline{z}}(z)}{f(z)-w_{0}}>0\ \text{%
for }|z|\leq r<1.
\end{equation*}

This is equivalent to

\begin{equation*}
\Re \dfrac{f(z)-w_{0}}{zf_{z}(z)-\overline{z}f_{\overline{z}}(z)}>0\text{\
for }|z|\leq r<1,
\end{equation*}

or

\begin{equation}
\Re \dfrac{f(z)}{zf_{z}(z)-\overline{z}f_{\overline{z}}(z)}>\Re \dfrac{w_{0}%
}{zf_{z}(z)-\overline{z}f_{\overline{z}}(z)}\text{\ for }|z|\leq r<1.
\end{equation}

It follows from (2.1) that

\begin{equation}
|f(z)|^{2}\Re \dfrac{zf_{z}(z)-\overline{z}f_{\overline{z}}(z)}{f(z)}%
>|w_{0}|^{2}\Re \dfrac{zf_{z}(z)-\overline{z}f_{\overline{z}}(z)}{w_{0}}%
\text{\ for }|z|\leq r<1.
\end{equation}

Now if $f(z)\in ST_{Lh}(\alpha )$, we have $e^{i\theta }f(e^{-i\theta }z)\in
ST_{Lh}(\alpha ).$ It follows that if $w_{0}\in \Omega _{r}$ then $\rho
e^{i\theta }\in \Omega _{r},$ with $\rho =|w_{0}|$ and $-\pi <\theta \leq
\pi .$

Therefore, if $w_{0}\in \Omega _{r},$ (2.2) must holds for all points $%
w=\left\vert w_{0}\right\vert e^{i\theta },$ $-\pi <\theta \leq \pi $ and so

\begin{equation*}
\left\vert f(z)^{2}\right\vert \Re \dfrac{zf_{z}(z)-\overline{z}f_{\overline{%
z}}(z)}{f(z)}\geq |w_{0}|\left\vert zf_{z}(z)-\overline{z}f_{\overline{z}%
}(z)\right\vert \text{\ for }|z|\leq r<1,
\end{equation*}

and hence

\begin{equation*}
\left\vert \dfrac{f(z)^{2}}{zf_{z}(z)-\overline{z}f_{\overline{z}}(z)}%
\right\vert \Re \dfrac{zf_{z}(z)-\overline{z}f_{\overline{z}}(z)}{f(z)}\geq
|w_{0}|\text{\ for }|z|\leq r<1.
\end{equation*}

We next consider the function

\begin{equation}
\Psi (f,z)=\left\vert \dfrac{f(z)^{2}}{zf_{z}(z)-\overline{z}f_{\overline{z}%
}(z)}\right\vert \Re \dfrac{zf_{z}(z)-\overline{z}f_{\overline{z}}(z)}{f(z)}%
\text{, }
\end{equation}

where $\ f\in ST_{Lh}(\alpha )$and \ $z$ is fixed, $|z|=r.$ Clearly, $%
\underset{\text{ }f\in ST_{Lh}(\alpha )}{\min }\Psi (f,z)$ is independent of
the choice of $z=re^{i\theta },$ $-\pi <\theta \leq \pi .$ Let $\left\vert
z\right\vert =r,$ $r>0,$ then $\lambda _{\alpha }(r)=\underset{\text{ }f\in
ST_{Lh}(\alpha )}{\min }\Psi (f,z)$ is the radius of $\Omega _{r}.$ Since $%
f\in ST_{Lh}(\alpha )$ by Theorem A we obtain%
\begin{equation}
\Re \dfrac{zf_{z}(z)-\overline{z}f_{\overline{z}}(z)}{f(z)}=\Re \dfrac{%
z\varphi ^{\prime }(z)}{\varphi (z)}=\Re ((1-\alpha )p(z)+\alpha ),
\end{equation}%
where $p\in P_{Lh}.$ From Theorem B, it follows that if $f\in ST_{Lh}(\alpha
)$ then 
\begin{equation}
|f(z)|\geq \dfrac{r}{(1+r)^{2\alpha }}\exp \left( \dfrac{-(1-\alpha )4r}{%
(1+r)}\right) .
\end{equation}%
Substituting (2.4) and (2.5) into (2.3), we get 
\begin{eqnarray*}
\Psi (f,z) &=&\left\vert \dfrac{f(z)}{\dfrac{zf_{z}(z)-\overline{z}f_{%
\overline{z}}(z)}{f(z)}}\right\vert \Re \dfrac{zf_{z}(z)-\overline{z}f_{%
\overline{z}}(z)}{f(z)} \\
&\geq &\dfrac{r}{(1+r)^{2\alpha }}\exp \left( \dfrac{-(1-\alpha )4r}{(1+r)}%
\right) \dfrac{\Re \left[ \alpha +(1-\alpha )p(z)\right] }{\ \left\vert 
\dfrac{1}{1-a}\left[ \alpha +(1-\alpha )p(z)\right] -\dfrac{\overline{a}}{1-%
\overline{a}}\left[ \alpha +(1-\alpha )\overline{p(z)}\right] \right\vert }
\\
&\geq &\dfrac{r}{(1+r)^{2\alpha }}\exp \left( \dfrac{-(1-\alpha )4r}{(1+r)}%
\right) \dfrac{\left[ \alpha +(1-\alpha )\dfrac{1-r}{1+r}\right] }{\
\left\vert \dfrac{1}{\left\vert 1-a\right\vert }\left[ \alpha +(1-\alpha
)|p(z)|\right] +\dfrac{\left\vert a\right\vert }{\left\vert 1-a\right\vert }%
\left[ \alpha +(1-\alpha )\left\vert p(z)\right\vert \right] \right\vert } \\
&\geq &\dfrac{r}{(1+r)^{2\alpha }}\exp \left( \dfrac{-(1-\alpha )4r}{(1+r)}%
\right) \dfrac{\left( \alpha +(1-\alpha )\dfrac{1-r}{1+r}\right) }{\left(
\alpha +(1-\alpha )\left( \dfrac{1+r}{1-r}\right) \right) \left( \dfrac{1+r}{%
1-r}\right) }.
\end{eqnarray*}

We set 
\begin{equation*}
\lambda _{\alpha }(r)=\dfrac{r}{(1+r)^{2\alpha }}\exp \left( \dfrac{%
-(1-\alpha )4r}{(1+r)}\right) \dfrac{\left( \alpha +(1-\alpha )\dfrac{1-r}{%
1+r}\right) }{\left( \alpha +(1-\alpha )\left( \dfrac{1+r}{1-r}\right)
\right) \left( \dfrac{1+r}{1-r}\right) }.
\end{equation*}
Then $\lambda _{\alpha }(r_{0})$ is the radius of $\Omega _{r},$ where $%
r_{0}\in (0,1)$ and $r_{0}$ is the smallest positive root of the equation $%
8r^{5}\alpha ^{3}-12r^{5}\alpha ^{2}+6r^{5}\alpha -r^{5}-16r^{4}\alpha
^{3}+12r^{4}\alpha ^{2}+4r^{4}\alpha -3r^{4}+8r^{3}\alpha ^{3}-36r^{3}\alpha
^{2}+32r^{3}\alpha -8r^{3}+4r^{2}\alpha ^{2}-4r^{2}\alpha +4r^{2}-6r\alpha
+9r-1=0.$

We note that$\underset{\text{ }f\in ST_{Lh}(\alpha )}{\min }\Psi (f,z)$ is
attained in $ST_{Lh}(\alpha )$ by a function of the form $f(z)=\overline{%
\eta }f_{0}(\eta z),$ where%
\begin{equation}
f_{0}(z)=\dfrac{z(1-\overline{z})}{(1-z)}\dfrac{1}{(1-\overline{z})^{2\alpha
}}\exp (1-\alpha )\Re \dfrac{4z}{1-z}.
\end{equation}
\end{proof}

For the particular case, where $f\in ST_{Lh}(0)~$we establish the following
corollary.

\begin{corollary}
Let $f\in ST_{Lh}(0)$, then $f$ is starlike with respect to any point of $%
\Omega _{r}$, where $\Omega _{r}\ $is a disk $\{w:|w|<\lambda _{\alpha
}(r_{0})\}\ ~$with$\ \lambda _{\alpha }(r_{0})=8.7462\times 10^{-2}.$
\end{corollary}

\begin{proof}
Let $U_{r}(f)=f(|z|\leq r<1),$ $w=f(z)\in ST_{Lh}(0).$ Proceeding in a
similar fashion as in the above proof, we show that $U_{r}(f)$ is starlike
with respect to $w_{0}$ if and only if $|w_{0}|\leq \Psi (f,z),\ \ $where 
\begin{equation*}
\Psi (f,z)=\left\vert \dfrac{f(z)}{\dfrac{zf_{z}(z)-\overline{z}f_{\overline{%
z}}(z)}{f(z)}}\right\vert \Re \dfrac{zf_{z}(z)-\overline{z}f_{\overline{z}%
}(z)}{f(z)}.
\end{equation*}

In particular$\ |w_{0}|\leq \lambda _{\alpha }(r)=\underset{\text{ }f\in
ST_{Lh}(0)}{\min }\Psi (f,z),$with $z$ fixed, $|z|=r$. Since$\ f\in
ST_{Lh}(0),$we have 
\begin{eqnarray*}
\lambda _{\alpha }(r) &\geq &r\exp \left( \dfrac{-4r}{(1+r)}\right) \left( 
\dfrac{1-r}{1+r}\right) ^{3} \\
&=&r\exp \left( \dfrac{-4r(1-r)^{3}}{(1+r)^{4}}\right) .
\end{eqnarray*}%
We can minimize $\lambda _{\alpha }(r)\ $by taking the smallest positive
root of the equation 
\begin{equation*}
r^{3}+3r^{2}+9r-1=0
\end{equation*}
which is $r_{0}=0.10715.$ Hence $\ \lambda _{\alpha }(r_{0})=8.7462\times
10^{-2}.$ We note that $\underset{\text{ }f\in ST_{Lh}(\alpha )}{\min }\Psi
(f,z)$ is attained in $ST_{Lh}(0)$ by a function of the form $f(z)=\overline{%
\eta }f_{0}(\eta z),$ where%
\begin{equation}
f_{0}(z)=\dfrac{z(1-\overline{z})}{(1-z)}\exp \Re \dfrac{4z}{1-z}.
\end{equation}
\end{proof}

\section{Close to starlike logharmonic mappings of order $\protect\alpha $}

In this section, we consider the set of all logharmonic mappings $F(z)\ $%
which can be factorized as \ the product of a logharmonic mapping $f(z)\in
ST_{Lh}(\alpha )$ with respect to $a\in B$ and a logharmonic mapping $%
R(z)\in P_{Lh}\ $with respect to the same $a$.

\begin{definition}
We say $F(z)$ is close to starlike of order $\alpha ,$ if $F(z)=f(z)R(z)$,
where \ $f$\ $\in $ $ST_{Lh}(\alpha )$ \ with respect to $a\in B$ and \ $%
R\in P_{Lh}$ with respect to the same $a.$ We denote by $CST_{Lh}(\alpha )$
the set of all close to starlike logharmonic mappings of order $\alpha .$
\end{definition}

Note that, if $\alpha =0,$ we get the class of close to starlike logharmonic
mappings and if $R(z)=1$ then $F\in ST_{Lh}(\alpha ).$

Close to starlike logharmonic mappings have the following geometrical
property: Under the mapping $F(z)$, the radius vector of the image of $%
|z|=r<1$, never turns back by an amount more than $\pi .$ Observe $F$ is not
necessarily \ univalent starlike on $U$. For example, take $F(z)=z(1-z),$
where $z\in ST(\alpha )$ and $1-z\in P.$

In the next result we determine the radius of univalence and starlikeness
for these mappings $F\in CST_{Lh}(\alpha )$.

\begin{theorem}
Let $F\in CST_{Lh}(\alpha ).$ Then $F$ maps the disk $|z|<\rho $ onto a
starlike domain, where $\ \rho \leq \dfrac{2-\alpha -\sqrt{\alpha
^{2}-2\alpha +3}}{1-2\alpha }\ \ $for $\alpha \neq \ \frac{1}{2},$ and $\rho
\leq \frac{1}{3}$ for $\alpha =\frac{1}{2}.$The upper bound is best possible
for all $a\in B$.

\begin{proof}
Let $F(z)=f(z)R(z)\in CST_{Lh}(\alpha )$, where \ $f$\ =$zh\overline{g}\in $ 
$ST_{Lh}(\alpha )$ \ with respect to $a\in B$ and \ $R=H\overline{G}\in
P_{Lh}$ with respect to the same $a$. $F(z)$ is logharmonic with respect to
the same $a$ and we have$\ $

\begin{equation}
\Re \dfrac{zF_{z}(z)-\overline{z}F_{\overline{z}}(z)}{F(z)}=\Re \dfrac{%
zf_{z}(z)-\overline{z}f_{\overline{z}}(z)}{f(z)}+\Re \dfrac{zR_{z}(z)-%
\overline{z}R_{\overline{z}}(z)}{R(z)}.
\end{equation}
From Theorem\ A, we have%
\begin{equation}
f(z)=\varphi (z)\exp 2\Re \int_{0}^{z}\dfrac{a(s)\varphi ^{\prime }(s)}{%
(1-a(s))\varphi (s)}ds,
\end{equation}
where $\varphi (z)=\dfrac{zh}{g}\in ST(\alpha ).$ Moreover, from [2], it
follows that 
\begin{equation}
R(z)=p(z)\exp 2\Re \int_{0}^{z}\dfrac{a(s)p^{\prime }(s)}{(1-a(s))p(s)}ds,
\end{equation}
where $p=\dfrac{H}{G}\in P$.

Substituting (3.2), (3.3) into (3.1), simple calculations lead to 
\begin{eqnarray*}
\Re \dfrac{zF_{z}(z)-\overline{z}F_{\overline{z}}(z)}{F(z)} &=&\Re \dfrac{%
zf_{z}(z)-\overline{z}f_{\overline{z}}(z)}{f(z)}+\Re \dfrac{zR_{z}(z)-%
\overline{z}R_{\overline{z}}(z)}{R(z)} \\
&=&\Re \left( \dfrac{z\varphi ^{\prime }(z)}{\varphi (s)}\right) +\Re \left( 
\dfrac{zp^{\prime }(z)}{p(z)}\right) .
\end{eqnarray*}%
But since 
\begin{equation*}
\Re \left( \dfrac{zp^{\prime }(z)}{p(z)}\right) \geq \dfrac{-2|z|}{1-|z|^{2}}%
\text{, and \ }\Re \left( \dfrac{z\varphi ^{\prime }(z)}{\varphi (s)}\right)
>(1-\alpha )\dfrac{1-|z|}{1+|z|}+\alpha ,
\end{equation*}%
we get 
\begin{equation*}
\Re \dfrac{zF_{z}(z)-\overline{z}F_{\overline{z}}(z)}{F(z)}\geq (1-\alpha )%
\dfrac{1-|z|}{1+|z|}+\alpha -\dfrac{2|z|}{1-|z|^{2}}=\dfrac{(1-2\alpha
)|z|^{2}+(2\alpha -4)|z|+1}{1-|z|^{2}}.
\end{equation*}%
Hence, $\Re \dfrac{zF_{z}(z)-\overline{z}F_{\overline{z}}(z)}{F(z)}>0$ if 
\begin{equation*}
\ (1-2\alpha )|z|^{2}+(2\alpha -4)|z|+1>0.
\end{equation*}%
In the case $\alpha =\frac{1}{2}$, the above is satisfied for $|z|<\frac{1}{3%
}$, so the radius of starlikeness is $\rho =\frac{1}{3}.$ For $\alpha \neq 
\frac{1}{2},\ $the radius of starlikeness $\rho $ is the smallest positive
root(less than 1) of \ $(1-2\alpha )\rho ^{2}+(2\alpha -4)\rho +1=0$ which
is $\dfrac{2-\alpha -\sqrt{\alpha ^{2}-2\alpha +3}}{1-2\alpha }.$ Therefore,
\ $F$ \ is univalent on $|z|<\dfrac{2-\alpha -\sqrt{\alpha ^{2}-2\alpha +3}}{%
1-2\alpha }$ \ and \ maps $\left\{ z:|z|<\dfrac{2-\alpha -\sqrt{\alpha
^{2}-2\alpha +3}}{1-2\alpha }\right\} $ onto a starlike domain. We consider
the analytic function $F(z)=\dfrac{z}{(1-z)^{2-2\alpha }}\dfrac{1+z}{1-z},$
where $f(z)=\dfrac{z}{(1-z)^{2-2\alpha }}\in ST(\alpha )\subset
ST_{Lh}(\alpha )$ and $p(z)=\dfrac{1+z}{1-z}\in P\subset P_{Lh}$. We have $%
F^{\prime }(\dfrac{2-\alpha -\sqrt{\alpha ^{2}-2\alpha +3}}{1-2\alpha })=0\ $%
for $\alpha \neq \frac{1}{2}$ and $F^{\prime }(\frac{-1}{3})=0\ $for $\alpha
=\frac{1}{2}.$ Hence, the upper bound is best possible for the class $%
ST_{Lh}(\alpha )$ and $P_{Lh}.$ Since $f(z)=zh\overline{g}$ $\in
ST_{Lh}(\alpha )$ if and only if $\varphi (z)=zh/g\in ST(\alpha )$ and $%
R(z)=H\overline{G}\in P_{Lh}$ if and only if $p=H/G\in P$ (see[2]), the same
bound is best possible for all $a\in B.$
\end{proof}
\end{theorem}

\begin{corollary}
Let $F\in CST_{Lh}(\alpha )$, then $F\in ST_{Lh}(\alpha )\ \ $in $|z|<\rho , 
$ for $\rho \leq \dfrac{2-\alpha -\sqrt{-2\alpha +3}}{1-\alpha }.$

\begin{proof}
$F\in ST_{Lh}(\alpha )$ if $\Re \dfrac{zF_{z}(z)-\overline{z}F_{\overline{z}%
}(z)}{F(z)}>\alpha .\ $Using the proof of the previous \ theorem, \ this
will be satisfied for 
\begin{equation*}
\dfrac{(1-2\alpha )|z|^{2}+(2\alpha -4)|z|+1}{1-|z|^{2}}>\alpha ,
\end{equation*}%
that is for$\ (1-\alpha )|z|^{2}+(2\alpha -4)|z|+1-\alpha >0.$ the radius of
starlikeness $\rho $ is the smallest positive root(less than 1) of $\
(1-\alpha )|z|^{2}+(2\alpha -4)|z|+1-\alpha =0$ which is $\dfrac{2-\alpha -%
\sqrt{-2\alpha +3}}{1-\alpha }.$
\end{proof}
\end{corollary}

\begin{theorem}
Let $F\in CST_{Lh}(\alpha )\ $with respect to a given a$.$ Let $f^{\ast }\in
ST_{Lh}(\alpha )$ with respect to the same $a.$ Then $Q(z)=F(z)^{\lambda
}f^{\ast }(z)^{1-\lambda }$, $0<\lambda <1,$ is univalent and starlike in 
\begin{equation*}
|z|<\dfrac{1+\lambda -\alpha -\sqrt{\alpha ^{2}-2\lambda \alpha +\lambda
^{2}+2\lambda }}{1-2\alpha }
\end{equation*}%
for $\alpha \neq \frac{1}{2},$and in 
\begin{equation*}
|z|<\frac{1}{2\lambda +1},
\end{equation*}%
for $\alpha =\frac{1}{2}.$ The bound is best possible for all $a\in B.$

\begin{proof}
Let $Q(z)=F(z)^{^{\lambda }\ }f^{\ast }(z)^{1-\lambda }$, $0<\lambda <1,$
where $F(z)=f(z)R(z),$ $f\in ST_{Lh}(\alpha )$, $R\in P_{Lh},$ and $f^{\ast
}\in ST_{Lh}(\alpha ).$ All functions are logharmonic with respect to the
same $a\in B.$Then $Q(z)$ is logharmonic with respect to the same $a.$
Moreover, we have 
\begin{eqnarray*}
&&\Re \dfrac{zQ_{z}(z)-\overline{z}Q_{\overline{z}}(z)}{Q(z)} \\
&=&\lambda \Re \dfrac{zf_{z}(z)-\overline{z}f_{\overline{z}}(z)}{f(z)}%
+\lambda \Re \dfrac{zR_{z}(z)-\overline{z}R_{\overline{z}}(z)}{R(z)}%
+(1-\lambda )\Re \dfrac{zf_{z}^{\ast }(z)-\overline{z}f_{\overline{z}}^{\ast
}(z)}{f^{\ast }(z)} \\
&\geq &\lambda \dfrac{(1-2\alpha )|z|^{2}+(2\alpha -4)|z|+1}{1-|z|^{2}}%
+(1-\lambda )\left( (1-\alpha )\dfrac{1-|z|}{1+|z|}+\alpha \right) \\
&=&\dfrac{(1-2\alpha )|z|^{2}+2(\alpha -\lambda -1)|z|+1}{1-|z|^{2}}.
\end{eqnarray*}%
Hence, $\Re \dfrac{zQ_{z}(z)-\overline{z}Q_{\overline{z}}(z)}{Q(z)}>0$ if 
\begin{equation*}
(1-2\alpha )|z|^{2}+2(\alpha -\lambda -1)|z|+1>0.
\end{equation*}%
For $\alpha =\frac{1}{2}$, the last inequality is satisfied for $|z|<\frac{1%
}{1+2\lambda }.$ Hence, $Q(z)$ is univalent in $|z|<\frac{1}{1+2\lambda }$
and maps that circle onto a starlike domain. For $\alpha \neq \frac{1}{2},$
the last inequality is satisfied for $|z|<\dfrac{1+\lambda -\alpha -\sqrt{%
\alpha ^{2}-2\lambda \alpha +\lambda ^{2}+2\lambda }}{1-2\alpha }.$ Hence, $%
Q(z)$ is univalent in $\ |z|<\dfrac{1+\lambda -\alpha -\sqrt{\alpha
^{2}-2\lambda \alpha +\lambda ^{2}+2\lambda }}{1-2\alpha }$ and maps that
circle onto starlike domain. We consider the function 
\begin{equation*}
Q(z)=F_{0}(z)^{^{\lambda }\ }f_{0}^{\ast }(z)^{1-\lambda },
\end{equation*}%
where 
\begin{equation*}
F_{0}(z)=\dfrac{z}{(1-z)^{2-2\alpha }}\frac{1-z}{1+z}
\end{equation*}%
and \ 
\begin{equation*}
^{\ }f_{0}^{\ast }(z)=\dfrac{z}{(1+z)^{2-2\alpha }}.
\end{equation*}%
$Q(z)$ satisfies the hypothesis of the theorem since $F_{0}(z)$ is a product
of \ analytic function which is starlike of order $\alpha $ \ and analytic
function with real part positive. Also, $^{\ }f_{0}^{\ast }(z)$ is starlike
analytic function of order $\alpha $ and therefore, it belongs to the set $%
ST_{Lh}(\alpha ).\ $Moreover, $Q^{\prime }(\dfrac{1+\lambda -\alpha -\sqrt{%
\alpha ^{2}-2\lambda \alpha +\lambda ^{2}+2\lambda }}{1-2\alpha })=0\ $for $%
\alpha \neq \frac{1}{2}$ and $Q^{\prime }(\frac{1}{1+2\lambda })=0\ $for $%
\alpha =\frac{1}{2}.$ From Theorem \ A, it follows the same bound is best
possible for all $a\in B.$
\end{proof}
\end{theorem}

For the particular case where $f^{\ast }\in ST_{Lh}(0),$ we have the
following theorem :

\begin{theorem}
Let$\ F\in CST_{Lh}(\alpha )\ $with respect to a given a$\ $and let $f^{\ast
}\in ST_{Lh}(0)$ with respect to the same $a.$ Then $Q(z)=F(z)^{\lambda
}f^{\ast }(z)^{1-\lambda }$, $0<\lambda <1,$ is univalent and starlike in $%
|z|<\dfrac{1+\lambda -\lambda \alpha -\sqrt{\lambda ^{2}\alpha ^{2}-2\lambda
^{2}\alpha +\lambda ^{2}+2\lambda }}{1-2\lambda \alpha }$for $\alpha \neq 
\frac{1}{2\lambda },$and in $|z|<\frac{1}{2\lambda +1},\ $for $\alpha =\frac{%
1}{2\lambda }.$ The bound is best possible for all $a\in B.$
\end{theorem}

\end{document}